\documentclass[11pt]{amsart}
\usepackage{amsmath}
\usepackage{amssymb}
\usepackage{tabularx}
\usepackage{enumerate}
\usepackage[dvipdfm]{graphicx}
\usepackage{texdraw}

\usepackage{mathrsfs}

\usepackage{amsfonts,amssymb,amsmath}
\usepackage{epsfig}
\usepackage{xcolor}
\makeatletter
\@addtoreset{equation}{section}

\makeatother
\newtheorem{thm}{Theorem}[section]
\newtheorem{lem}[thm]{Lemma}

\newtheorem{prop}[thm]{Proposition}
\newtheorem{remark}[thm]{Remark}

\newcommand{\R}{\mathbb{R}}
\newcommand{\Z}{\mathbb{Z}}

\newcommand{\D}{\mathcal{D}}

\newcommand{\X}{\mathscr{X}}

\newcommand{\OM}{\omega}

\begin{document}
\title[Regularity of Degenerate Diffusion Equations]{Regularity of Solutions to One-Dimensional Degenerate Diffusion Equations with Reactions$^\S$}
\thanks{$\S$ This research was partly supported by the NSFC (No. 12471199, 12101413, 12071299). }
\author[B. Lou  and J. Lu]{Bendong Lou$^\dag$ and Junfan Lu$^{\dag, *}$}
\thanks{$\dag$ Mathematics and Science College, Shanghai Normal University, Shanghai 200234, China.}
\thanks{{\bf Emails:} {\sf lou@shnu.edu.cn} (B. Lou), {\sf jlu@shnu.edu.cn} (J. Lu)}
\thanks{$*$ Corresponding author.}
\date{}

\begin{abstract}
We study the one-dimensional reaction-diffusion equation
\[
u_t=[A(u)]_{xx}+f(x,u).
\]
The diffusion operator belongs to a broad class of nonlinear degenerate diffusion operators that includes the porous medium operator as a special case.

We develop a systematic regularity theory for the solutions and their free boundaries.
First, we establish the $C^1$ regularity of the pressure variable $v$, together with a lower bound for its second spatial derivative.
Next, we prove Darcy's law and that, after the waiting time, a right (resp.\ left) free boundary moves with strictly positive (resp.\ negative) velocity (Theorem \ref{thm:Darcy0}), thereby strengthening the previously known results which only gave nonnegativity (resp.\ nonpositivity).
Finally, under additional structural assumptions on the diffusion and reaction terms, we obtain higher regularity for both the solution and its free boundaries (Theorems \ref{thm:k-th order} and \ref{thm:r-smooth}).
These results extend several classical regularity properties of the porous medium equation to a much broader class of degenerate diffusion equations with reactions.
\end{abstract}

\maketitle

\section{Introduction}

We consider the one-dimensional reaction diffusion equation
\begin{equation}\label{E}
u_t = [A(u)]_{xx} + f(x,u), \qquad x\in \R,\ t>0,
\tag{E}
\end{equation}
where the nonlinear diffusion term $A$ satisfies the following conditions:
\begin{equation*}\label{ass-A}
\left\{
 \begin{array}{l}
 A\in C^1([0,\infty)) \cap C^\infty ((0,\infty)), \ A(0)=A'(0)=0,\\
 A(u), A'(u), A''(u)>0 \mbox{ and } -A'(u) A'''(u) \leq K_0 [A''(u)]^2 \mbox{ for }u>0,\\
\int_0^1 \frac{A'(u)}{u} du<\infty,\ \frac{uA''(u)}{A'(u)} \to \OM\ \mbox{ as } u\to 0^+.
 \end{array}
 \right.
\tag{A}
\end{equation*}
Here $K_0$ and $\OM$ are positive real numbers. Examples of such $A$ include $A_1(u):=u^m\ (m>1)$ (which corresponds to the standard porous medium equation), $A_2(u) :=u^m+u^n\ (m, n>1)$, $A_3(u)= \int_0^u re^{\alpha r^\beta} dr\ (\alpha> 0,\ \beta >0)$ etc.
In any of these cases, the diffusion coefficient degenerates at $u=0$, which leads to finite speed of propagation and the appearance of free boundaries.

Degenerate diffusion equations arise naturally in population dynamics, combustion theory, porous media flow, phase transition phenomena and related applications (cf. \cite{A1,GN,GurMac,NS,SGM,Vaz-book,WuYin-book}).
The porous medium equation (PME, for short) is the prototype model and has been studied extensively. For the pure diffusion problem (without reaction term), a comprehensive theory is available, including existence and uniqueness of solutions, finite propagation, waiting-time phenomena, interface dynamics, regularity of the pressure variable (see \eqref{def-v} below), and asymptotic behavior (see, for example,  the monograph of V\'{a}zquez \cite{Vaz-book}).

The reaction term $f(x, u)$ in \eqref{E} is assumed to satisfy
\begin{equation*}\label{F}
f(x,u)\in C^2(\R^2),\quad f(x,0)\equiv 0,\quad f(x, u) \mbox{ is Lipschitz continuous in } u.
\tag{F}
\end{equation*}
The $C^2$ smoothness is mainly used to derive the a priori estimates.
When a reaction term is present, the interaction between nonlinear diffusion and local growth mechanisms
creates new difficulties. Although PMEs with reactions have attracted much attention \cite{ACP, AV, DQZ, Garriz, GK, LouZhou, PV, She}, many fundamental regularity questions remain poorly understood.
In particular, the precise motion of free boundaries, and higher  smoothness properties have not been systematically investigated for general reaction-diffusion equations with degenerate diffusion.

The present work is motivated partly by the previous work \cite{LouZhou}, where several free boundary properties (for the solutions to PMEs with reactions) were established as auxiliary tools, and the long-time behavior of solutions was the main issue studied there.
The purposes of this paper differ from those of \cite{LouZhou}. Instead of asymptotic behavior, we focus on the regularity theory itself. Moreover, we replace the porous medium diffusion with a much broader class of degenerate diffusion operators satisfying \eqref{ass-A}.

The principal contributions of this paper are summarized as follows.

\begin{enumerate}
\item We work with a general class of degenerate diffusion operators satisfying \eqref{ass-A},     extending the porous-medium framework to a much broader class.
%
%\item We establish refined regularity estimates for the pressure variable, including the boundedness of its second spatial derivative.

\item We prove that, once the waiting time has ended, each right (resp. left) free boundary moves with strictly positive (resp. negative) velocity (see details in Theorem \ref{thm:Darcy0}). This improves the previously known conclusion (for PMEs with reactions) that the interface velocity is merely nonnegative/nonpositive.

\item Under additional structural assumptions, we derive higher  regularity for both the pressure variable (Theorem \ref{thm:k-th order}) and the free boundaries (Theorem \ref{thm:r-smooth}), yielding a systematic regularity theory for free boundaries.
\end{enumerate}

The remainder of the paper is organized as follows.  Section 2 contains preliminary material concerning the pressure transformation and basic properties of free boundaries. Section 3 establishes Darcy's law and shows the positive/negative velocity for each right/left free boundary after the waiting time (Theorem \ref{thm:Darcy0}). Section 4 is devoted to higher regularity of the solution and its free boundaries (Theorems \ref{thm:k-th order} and \ref{thm:r-smooth}).

\section{The Pressure Variable and its Basic Properties}
In this preliminary section, we define a generalized pressure function $v$. As in the study of the PME, working with this function proves far more convenient than dealing directly with the original density function $u$.
We also present several fundamental properties of the pressure function, including $C^1$ regularity and lower bound of its second-order spatial derivatives.
These properties are not completely new; especially, they are well known for the pure PME and are by now fairly well understood for PMEs with reaction terms; we show that they remain valid for general reaction-diffusion equations with degenerate diffusion. For the sake of later reference, we record these results here and briefly outline the underlying ideas of the proofs.

\subsection{Generalized Pressure Variable}
In this paper, we apply a generalized pressure function $v(x,t)$, which is defined by
\begin{equation}\label{def-v}
v(x,t) := \Psi(u(x,t)) := \int_0^{u(x,t)} \frac{A'(r)}{r} dr\qquad \mbox{ for\ \ }u\geq 0.
\end{equation}
Denote 
$$
v_{\infty}:= \int_0^{\infty} \frac{A'(r)}{r} dr \in (0, \infty],\qquad \D_v^0 :=(0, v_{\infty})\quad \mbox{and}\quad \D_v := [0, v_{\infty}).
$$
Then the function $\Psi(u)$ maps monotonically from $[0, +\infty)$ to $\D_v$. 
So we can define the inverse function of $\Psi$ as follows
$$
u=\psi(v):=\Psi^{-1}(v),\quad v\in \D_v . 
$$
Clearly,
$$
\Psi (0)=\psi (0)=0,\quad \psi(v)>0 \mbox{ for } v\in \D_v^0, \quad \Psi(u)>0 \mbox{ for } u>0, 
$$
$\Psi\in C([0,\infty))\cap C^\infty ((0,\infty))$, and $\psi \in C(\D_v )\cap C^\infty (\D_v^0 )$, 

Denote
\begin{equation}\label{def-B}
B(v) := A'(\psi(v)),\qquad v\in \D_v ,
\end{equation}
and
\begin{equation}\label{def-h}
h(x,v) := \left\{
 \begin{array}{ll}
\displaystyle \frac{f(x,\psi(v))}{\psi'(v)}, & x\in \R,\ v\in \D_v ,\\
0,& x\in \R,\ v=0.
\end{array}
\right.
\end{equation}
Then we can use $v$ to rewrite the Cauchy problem for \eqref{E} with initial data $u_0$ as
\begin{equation}\label{p-v}
\left\{
\begin{array}{ll}
v_t = B(v) v_{xx} + v_x^2 + h(x,v),&  x\in \R,\ t>0,\\
v(x,0)= v_0(x) := \Psi(u_0(x)), & x\in \R.
\end{array}
\right.
\end{equation}

Based on their definitions, we state the smoothness properties of $B$ and $h$.

\begin{lem}\label{lem:smooth-Ah-1}
Assume \eqref{ass-A} and \eqref{F}. Then $B(v)\in C^1(\D_v )$ and $h(x,v)\in C^1 (\R\times \D_v )$.
\end{lem}

\begin{proof}
By the definition of $B$, we have
$$
B(0)=0 \mbox{\ \ and\ \ } B(v) =\frac{A'(\psi(v))}{\psi(v)} \cdot \psi(v) = \frac{\psi(v)}{\psi'(v)} \mbox{ for } v\in \D_v^0 .
$$
Note that, by \eqref{ass-A} there holds,
\begin{equation}\label{B'=A*}
B'(0^+) =\lim\limits_{v\to 0^+} \frac{A'(\psi(v))}{v} =
\lim\limits_{u\to 0^+} \frac{A'(u)}{\Psi(u)} =
\lim\limits_{u\to 0^+} \frac{uA''(u)}{A'(u)} = \OM >0.
\end{equation}
Hence $B\in C^1(\D_v )\cap C^\infty(\D_v^0 )$ and $B'(v)>0$ for $v\in \D_v $.

Using a similar derivation as in \eqref{B'=A*} one can show that
$$
h_v(x,0^+) = \lim\limits_{v\to 0^+} \frac{f(x,\psi(v))}{v \psi'(v)}= \lim\limits_{u\to 0^+} \frac{f(x, u)}{u} \frac{A'(u)}{\Psi(u)}  = f_u(x,0^+) \OM,\qquad x\in \R.
$$
On the other hand, when $v> 0$, by $h(x,v)= \frac{f(x,\psi(v))}{\psi'(v)}= \frac{f(x, u)A'(u)}{u}$ we have
$$
h_v(x, v)=  f_u(x, u)- \frac{f(x, u)}{u}+ \frac{f(x,u)}{u}\frac{A''(u)u}{A'(u)},
$$
which implies
$$
\lim_{v\to 0^+}h_v(x, v)=\lim_{u\to 0^+}f_u(x, u)- \frac{f(x, u)}{u}+ \frac{f(x,u)}{u}\frac{A''(u)u}{A'(u)} = f_u(x,0^+) \OM.
$$
Hence, $h\in C^1(\R\times \D_v )\cap C^2(\R\times \D_v^0 )$.
\end{proof}

In the last part of the paper, we will investigate the higher  regularity of $v$. To this end, we require higher  regularity for $B$ and $h$, which can be achieved by imposing structural conditions on $A$ and $f$ as follows: for some positive integer $k$, there holds
\begin{equation}\label{ass-H}
\left\{
\begin{array}{l}
A(u)\in C^p ([0,\infty)) \mbox{ for } p> \OM (k+1), \ r^{-1} A'(r^{\frac{1}{\OM}}) \in C^k ([0,\infty)), \  \displaystyle\lim_{r\to 0+} r^{-1} A'(r^{\frac{1}{\OM}})> 0, \\
f(x,u)\in C^q (\R\times [0,\infty)) \mbox{ for } q>\OM(k-1)+1,\quad f(x,r^{\frac{1}{\OM}})r^{\frac{\OM-1}{\OM}}\in C^k(\R\times [0,\infty)),
\end{array}
\right.
\tag{H}
\end{equation}
where $\OM$ is the positive number in \eqref{ass-A}, $p$ and $q$ are positive integers. Let us elaborate further on these conditions. When $A(u)\in C^p ([0,\infty))$ and $f(x,u)\in C^q (\R\times [0,\infty))$, we have the following Taylor's formulas:
$$
\left\{
\begin{array}{l}
A(u) = \frac12 A''(0)u^2 + \cdots + \frac{1}{p!} A^{(p)}(0) u^p + o(u^p),\\
f(x,u) = f_u(x,0) u + \frac12 f_{uu}(x,0)u^2 + \cdots + \frac{1}{q!} \frac{\partial^{q} f(x,0)}{\partial u^q}  u^q + o(u^q),
\end{array}
\right.
$$
and so the positive number $\OM$ in \eqref{ass-A} is nothing but $n-1$ for some $n\leq p$ given by
$$
A^{(j)}(0)=0\ (0\leq j<n) ,\quad A^{(n)} (0)>0.
$$
In this case, the other conditions in \eqref{ass-H} hold for $A$ and $f$ provided
$$
\begin{array}{l}
p>(k+1)(n-1),\quad A^{(j)}(0)=0 \mbox{ when } \frac{j-1}{n-1} \not\in \mathbb{N},\\
q>(k-1)(n-1)+1,\quad \frac{\partial^j }{\partial u^j} f (x,0) =0 \mbox{ when } \frac{j-1}{n-1} +1 \not\in \mathbb{N}.
\end{array}
$$

\begin{lem}\label{lem:smooth-Ah-2}
Assume  \eqref{ass-H}. Then $B(v)\in C^k (\D_v )$ and $h(x,v)\in C^k (\R\times \D_v )$.
\end{lem}

\begin{proof}
From \eqref{ass-H} we have the auxiliary functions
$$
\Phi(r):=r^{-1}A'(r^{1/\OM})\in C^k([0,\infty)),\qquad
\Lambda(x,r):=f(x,r^{1/\OM})r^{(\OM-1)/\OM}\in C^k(\mathbb{R}\times[0,\infty)),
$$
and the assumption guarantees $\displaystyle\Phi(0)=\lim_{r\to0^+}r^{-1}A'(r^{1/\OM})> 0$.
Define
$$
G(r):=\frac{1}{\OM}\int_0^r\Phi(y)\,dy,\qquad r\ge0.
$$
Then $G\in C^{k+1}([0,\infty))$ and $G'(r)=\Phi(r)/\OM>0$. Recall $\int_0^\infty\frac{A'(u)}{u}du= v_{\infty}$, we have $G(r)\to v_{\infty}$ as $r\to\infty$, thus $G$ is strictly increasing and is a $C^{k+1}$-diffeomorphism from $[0,\infty)$ to $\D_v $. Its inverse
$$
g:= G^{-1}:\D_v \to[0,\infty)
$$
also belongs to $C^{k+1}(\D_v )$ and satisfies $g(0)=0$.

Now recall the pressure variable \eqref{def-v}:
$v=\Psi(u)=\int_0^u\frac{A'(s)}{s}ds$.
Substituting $r=s^\OM$ and using $A'(s)=s^\OM\Phi(s^\OM)$ we obtain
$$
v=\int_0^u s^{\OM-1}\Phi(s^\OM)\,ds =\int_0^{u^\OM}\frac{\Phi(r)}{\OM}\,dr = G(u^\OM).
$$
Consequently $u^\OM= g(v)$ and
\begin{equation}\label{u-G}
u=\psi(v)=[g(v)]^{1/\OM}.
\end{equation}
We can now express $B(v)$ and $h(x,v)$ directly in terms of $g$ and the auxiliary functions.
By \eqref{def-B} and \eqref{ass-A},
$$
B(v)= A'(\psi(v))= A'(u)= u^\OM\Phi(u^\OM)= g(v)\,\Phi(g(v)).
$$
Since $g\in C^{k+1}(\D_v )$ and $\Phi\in C^k([0,\infty))$, both
$g(v)$ and $\Phi(g(v))$ are $C^k$ functions of $v$.
Therefore $B(v)\in C^k(\D_v )$.

For the reaction term, we use the identity $h(x,v)=\frac{f(x,u)A'(u)}{u}$. Substituting $u$ and $A'(u)$ gives
$$
h(x,v)=\frac{f(x,u)}{u}\,u^\OM\Phi(u^\OM)
      =\Lambda(x,u^\OM)\,\Phi(u^\OM)
      =\Lambda(x,g(v))\,\Phi(g(v)),
$$
where $g\in C^{k+1}(\D_v )$,
$\Lambda\in C^k(\mathbb{R}\times[0,\infty))$ and $\Phi\in C^k([0,\infty))$. Consequently $h(x,v)\in C^k(\mathbb{R}\times \D_v )$.
This completes the proof of Lemma~\ref{lem:smooth-Ah-2}.
\end{proof}

\subsection{Definition of Very Weak Solutions}
Throughout this paper we consider the Cauchy problem of \eqref{E} with initial data
chosen from the following set.
\begin{equation}\label{cond of initial}
%\X:= \big\{ w\in C(\R) \; \left| \;
% \exists\ b>0 \mbox{ such that }  w(x) >0 \mbox{ for } |x|<b, \mbox{ and } w(x) =0 \mbox{ for } |x|\geq b\right. \big\}.
\X:= \left\{ w \; \left| \;
\begin{array}{l}
w(x)\in  C(\R)\cap L^{\infty}(\R),  \mbox{ there exist } \Z_0 \subset \Z \mbox{ and real numbers }\\
\{l^0_i\}_{i\in \Z_0}, \ \{r^0_i\}_{i\in \Z_0}  \mbox{ with } \cdots\leq l^0_i < r^0_i \leq l^0_{i+1}<r^0_{i+1}\leq \cdots \\
\mbox{such that } w(x) >0 \mbox{ in } (l^0_i, r^0_i) \mbox{ and } w(x) =0 \mbox{ otherwise}.
\end{array}
\right. \right\}.
\end{equation}
For any $T>0$, denote $Q_T:= \R\times (0, T]$. A function $u(x, t)\in C(Q_T)\cap L^{\infty}(Q_T)$ is called a {\it very weak solution} of \eqref{E} with initial data $u_0(x)\in \X$ if, for any $\phi\in C^\infty_c(\R\times[0,T])$, there holds
$$
\int_{\R} u(x, T)\phi(x, T) dx- \int_{\R} u_0(x)\phi(x, 0) dx= \iint_{Q_T} f(x, u)\phi dx dt+ \iint_{Q_T} [u\phi_t+ A(u) \phi_{xx}] dx dt.
$$

For each given $u_0\in \X$, the well-posedness of the solutions to \eqref{E} with $u(x,0)=u_0(x)$, and its basic properties can be found in \cite{ACP, PV, LouZhou, Sacks, Vaz-book, WuYin-book} etc. (Some of them are well known for PMEs with or without reactions, and can be derived similarly for general degenerate diffusion equation \eqref{E}, see, for example, \cite{Knerr, Vaz-book,WuYin-book} and references therein). Using the pressure variable $v$, some basic properties are summarized as follows.

\begin{enumerate}[\noindent (a).]
\item {\it Global well-posedness}. Due to $|f(x,u)|\leq K_1 u$ for $u>0$ and some $K_1 >0$, the equation
    \eqref{E} with initial data $u_0\in \X$ has a time global very weak solution $u(x,t)$, which satisfies
    \begin{equation}\label{bdd-v}
    0\leq u(x,t) \leq  \|u_0 \|_{L^\infty } \cdot e^{K_1 t} ,\quad t>0.
    \end{equation}
(If $f(x,u)\leq 0$ for large $u$, say, for $u\geq 1$, then the upper bound of $u$ can be time-independent.)

\item {\it Free boundaries and their waiting times}. Some right free boundaries $r_i (t)$ and left ones $l_i(t)$ appear in the solution; $u(x,t)$ is positive and classical in each $(l_i(t),r_i(t))$, and vanishes on the complement of the union of these intervals. Each right free boundary $r_i(t)$ has a waiting time $t_*(r^0_i) \in [0,\infty]$ at the beginning stage. As for PME without or with a reaction (cf. \cite{Vaz-book, LouZhou}), $t_*(r^0_i)>0$ really happens for some initial data, under the additional condition
    $$
    -A'(u) A'''(u)\leq K_0 [A''(u)]^2 \mbox{ for }u>0
    $$
    in \eqref{ass-A}. (In fact, this condition is satisfied automatically for the PME, and is used only in this argument.) On the other hand, $t_*(r^0_i)=\infty$ is also possible due to the presence of the reaction term $f(x,u)$. This happens in particular in the case where $f$ is a bistable reaction, since the problem may have stationary solutions with compact support in such a case (see, for example, \cite[Lemmas 2.4 and 6.1]{LouZhou} for reaction-PME.)

    It is more interesting to consider the properties of the solution and its free boundaries after the waiting time, i.e., when $t_*(r^0_i)\in(0,\infty)$. It is known that, after the waiting time, $r_i (t)$ moves rightward monotonically: $r'_i (t)\geq 0$. Any left free boundary behaves similarly. One of our main aims in this paper is to show that each boundary satisfies Darcy's law, and its velocity after the waiting time is {\it strictly} positive or negative (see Section~3 for details).

\item {\it Positivity persistence and regularity}. For any $x_1\in \R$, once $u(x_1,t_1)>0$ for some $t_1\geq 0$, then $u(x_1,t)>0$ for all $t\geq t_1$, and so $u(x,t)$ is classical in a neighborhood of $(x_1,t_1)$.
\end{enumerate}

\subsection{Bounds of $v_x, v_t$ and Lower Bound of $v_{xx}$}
This subsection lists some basic regularity results which are useful in the sequel. They are well known for PME (especially, PME without reactions), but they do not seem to have been explicitly examined in the literature for general reaction-diffusion equations with degeneracy.

Let \(T>0\) be fixed. By the \(L^\infty\)-estimate in \eqref{bdd-v} we have the following bound for $v$:
\begin{equation}\label{bdd-v-new}
0\le v(x,t)\le V_T := \Psi \left(\|u_0\|_{L^\infty} e^{K_1T}\right)< v_{\infty}, \qquad (x,t)\in Q_T := \R\times (0,T].
\end{equation}
In what follows, we use the notation
\begin{equation}\label{def-HT}
B_T :=\|B(v) \|_{C^1([0,V_T])},\qquad H_T := \|h\|_{L^\infty (D_T)} + \|h_x\|_{L^\infty (D_T)} + \|h_v\|_{L^\infty (D_T)},
\end{equation}
where $D_T := \mathbb R \times [0,V_T]$.

The first estimate concerns the spatial gradient of \(v\).

\begin{prop}[Bound of $v_x$]\label{prop:vx}
Assume \eqref{ass-A} and \eqref{F}. Let $v$ be the solution of \eqref{p-v} in $(0,T]$ with initial data $v_0=\Psi(u_0)$, $u_0\in\X$.
\begin{enumerate}[\rm (i).]
\item If it is positive in \( (a,b)\times(0,T] \), then for every \( 0<\delta<\frac{b-a}{2} \) and \( 0<\tau<T \) there exists a constant \( M_1 =M_1 (V_T,B_T,H_T,\delta,\tau) \) such that
\begin{equation}\label{vx-est}
|v_x(x,t)| \le M_1, \qquad (x,t)\in [a+\delta,b-\delta]\times[\tau,T].
\end{equation}

\item If $v'_0(x)$ exists a.e. in $(a,b)$ and $M^0_{1ab} := \sup\limits_{[a,b]}|v'_0(x)|<\infty$, then the same estimate holds in $[a+\delta,b-\delta]\times (0,T]$ for $M_1 = M_1(V_T, B_T, H_T, \delta, M^0_{1ab})$.

\item If $(a,b)=\R$, $v'_0(x)$ exists a. e. in $\R$ and $M_1^0 := \sup\limits_{\R}|v'_0(x)|<\infty$, then the estimate holds in $\R\times (0, T]$ for $M_1= M_1 (V_T,B_T,H_T,M_1^0)$.
\end{enumerate}
\end{prop}

\begin{proof}
We only need to consider (i), since (ii) and (iii) are treated similarly. The argument is based on a Bernstein-type method introduced by Aronson \cite{A1969}.
The main difference from the classical PME is the presence of the variable coefficient \(B(v)\) and the reaction term \(h(x,v)\).

We outline the main ideas. Define $\phi: [0,1]\to [0,V_T]$ by
\[
\phi(r) :=\frac{V_T}{3} r(4-r),\qquad r\in [0,1].
\]
Then $\phi$ is strictly increasing in $[0,1]$. Introducing the transformation $w:= \phi^{-1} (v)$
and applying the Bernstein procedure to \( z :=\zeta^2 |w_x|^2 \) for some truncation function $\zeta$ which takes $1$ over $[a+\delta, b-\delta]\times [\tau, T]$, one obtains the following inequality
at any interior maximum point of \(z\) (if it exists)
\[
\frac23 V_T \zeta^2 |w_x|^4 \le C_1 \zeta |w_x|^3 + C_1 |w_x|^2 ,
\]
for some $C_1$ depending on $V_T, B_T, H_T,\delta$ and $\tau$.
Using Young's inequality yields $ \zeta^2|w_x|^2\le C_2$, and therefore
\[
|v_x| = \phi'(w)|w_x| \le C.
\]
The details are completely analogous to the Bernstein estimates in \cite{A1969}, with additional lower-order terms controlled by \eqref{def-HT}.
\end{proof}

Next, we study the lower bound for $v_{xx}$, which will play an important role in the sequel. The next result may be viewed as an Aronson-B\'{e}nilan type estimate (cf. \cite{AB}) for the general reaction-diffusion equation with degenerate diffusion \eqref{E}.

\begin{prop}[Aronson-B\'{e}nilan type estimate]\label{prop:AB}
Assume \eqref{ass-A} and \eqref{F}. Assume further that \eqref{ass-H} holds for $k\geq 2$.
Let \(v \) be the solution of \eqref{p-v} with initial data $v_0=\Psi(u_0)$, $u_0\in\X$, such that $v_0\in C^2$ a.e. in $\R$,
\begin{equation}\label{ass-v0xx}
0\leq v_0(x)\leq M_0^0,\quad |v_0'(x)|\le M^0_1, \quad \mbox{and}\ \ v_0''(x) \ge -M^0_2\ \ a.e.\ x\in \R.
\end{equation}
Then, for any $T>0$,  there exist positive constants \(D\) and \(\tau\), depending only on
$V_T, B_T, M_1^0,$ $M_2^0,$ $ H_{2T}$  such that
\[
v_{xx}(x,t) \ge -D(t+\tau),\quad (x,t)\in Q_T:= \R \times (0, T],\  \ \
\mbox{in the distribution sense},
\]
 where
$$
H_{2T} := H_T + \|h_{xx}\|_{L^\infty (D_T)} + \|h_{xv}\|_{L^\infty (D_T)} + \|h_{vv}\|_{L^\infty (D_T)},
\quad D_T:= \R\times [0,V_T].
$$
\end{prop}

\begin{proof}
We first assume $v>0$ in $Q_T$. In this case, $v$ is a classical solution and so $v_x$ is bounded by Proposition \ref{prop:vx}.
Let \(\eta :=v_{xx} \). Differentiating the equation in \eqref{p-v} twice in \(x\) yields
\begin{equation}\label{equ-eta}
\eta_t =   a(x, t) \eta_{xx} + b(x, t) \eta_x + c(x, t) \eta + d(x, t) \eta^2 + e(x, t),
\end{equation}
with $a= B(v)$,
$$
b = 2\big[ B'(v) + 1 \big] v_x ,\ \ c= B''(v) v_x^2 + h_v,\ \ d= B'(v) + 2, \ \
e =  h_{xx} + 2 h_{xv} v_x + h_{vv} v_x^2.
$$
The structural assumption \eqref{ass-H}, together with Lemma \ref{lem:smooth-Ah-2}, implies that all of the coefficients are bounded. Choosing
\[
D := \|e\|_{L^\infty (D_T)} + \frac{\|c\|_{L^\infty (D_T)}}{2} + 1 \quad \mbox{and}\quad \tau := \frac{M^0_2}{D} + 1,
\]
and combining $d(x, t)> 2$, we see that $-D(t+\tau)$ is a subsolution to the equation \eqref{equ-eta}. In fact, we have
\begin{eqnarray*}
&-c(x,t) D(t+\tau) + d(x,t) D^2 (t+\tau)^2 + e(x,t)+ D \\
& \geq  \left[2D \tau- \|c\|_{L^\infty (D_T)}\right]D(t+\tau)+ \left(D- \|e\|_{L^\infty (D_T)}\right)> 0 \qquad \mbox{ for } t> 0,
\end{eqnarray*}
while at $t=0$, there holds
$\eta(x,0) = v_0''(x) \ge -M_2^0 \ge -D\tau$.
So the comparison principle gives
\[
\eta(x,t)\ge -D(t+\tau),\qquad (x,t)\in Q_T.
\]

For general nonnegative solution $v\geq 0$, we consider the approximate problem: the equation of $v$ in \eqref{p-v} with initial data
$$
v(x,0)= v_0(x)+ \varepsilon, \qquad  x\in \R
$$
for any $0<\varepsilon \ll 1$. Denote the corresponding solution by $v_{\varepsilon}(x, t)$. Then we have
$$
0< v_{\varepsilon}\leq V_T+1,\quad |v_{\varepsilon x}|\leq \tilde{M}_1(V_T,B_T,H_T,M_1^0).
$$
Both of the bounds are independent of $\varepsilon$. The previous derivation when $v>0$ then implies that
$$
v_{\varepsilon xx}\geq -\tilde{D}(t+\tilde{\tau}),\quad (x,t)\in Q_T,
$$
where $\tilde{D}$ and $\tilde{\tau}$ depend on $V_T$, $B_T$, $H_{2T}$ and $\tilde{M}_1$.
Since $v_{\varepsilon}$ is strictly decreasing in $\varepsilon$, we can then take limit as $\varepsilon \to 0$, in the distributional sense, to obtain the desired conclusion for the case $v\geq 0$
(cf. \cite{AB}, \cite[\S 9.3]{Vaz-book} etc.).
\end{proof}

\begin{remark}\rm
We briefly clarify the regularity condition imposed on $v_0$. Several formulations are candidates of our assumption: (1) $v_0\in C^2(\R)$ and $v''_0(x)\geq -M_2^0$ in $\R$; (2) $v_0\in C^2$ a.e. in $\R$, and $v''_0(x)\geq -M_2^0$ a.e. in $\R$;
(3) the second weak derivative $D^2 v_0$ of $v_0$ exists, with $D^2 v_0\geq -M_2^0$ a.e. in $\R$ (i.e., $D^2 v_0 \in L^\infty(\R)$). In problem \eqref{p-v} we aim to accommodate typical initial data such as $\tilde{v}_0(x) = (1-x^2)_+$. Clearly, this initial datum fails to satisfy (1) and (3) --- since $D^2 \tilde{v}_0 =\pm \delta(x\pm 1)$ at $x=\mp 1$ --- but it does satisfy (2). Consequently, we adopt condition (2) in our proposition.
\end{remark}

Finally, we establish the estimate of $v_t$.

\begin{prop}[Bound of $v_t$]\label{prop:vt}
Under the hypotheses in the previous proposition, there holds
\[
|v_t(x,t)| \le C_1t+C_2, \qquad (x,t)\in Q_T,
\]
for some positive $C_1$ and $C_2$ depending on $V_T, B_T, M_1^0, M_2^0, H_{2T}$.
\end{prop}

\begin{proof}
The lower bound of $v_t$ follows directly from Proposition \ref{prop:AB} and the equation of $v$.

For the upper bound, introduce the auxiliary quantity
\[
P (x,t) :=v_t+ \frac{B_T+1}{2} v_x^2.
\]
A direct computation shows that \(P\) satisfies the following inequality
\[
P_t \le B(v)P_{xx}  +  2v_xP_x + \frac1{B(v)} \Bigl( - P^2 + b_1 P + c_1 \Bigr) + d_1,
\]
where  \( b_1, c_1, d_1 \) depend only on the structural constants.

We first consider the case $v>0$ as in the proof of Proposition \ref{prop:AB}. In this case \(v_x\) and other coefficients in the above inequality are bounded. By the maximum principle we have
\[
P(x,t) \le C_1t+C_2,
\]
provided $C_1$ and $C_2$ are chosen sufficiently large. This reduces to the desired estimate for \(v_t\).

For general nonnegative solution $v\geq 0$, the estimate follows by approximation with strictly positive solutions and passage to the limit in the distributional sense, as in the proof of the previous proposition.
\end{proof}

\section{Darcy's Law, Positive/Negative Velocity of a Free Boundary}

In this section we establish the key free-boundary estimate. We show that once the waiting-time regime ends, a right (resp. left) free boundary moves with strictly positive (resp. negative) velocity and satisfies Darcy's law.

For simplicity of the presentation, in the rest of the paper we consider simpler initial data chosen from
\begin{equation}\label{def-X0}
\X_0 := \left\{ w\in C(\R) \; \left| \;
 \begin{array}{l}
 \exists\ b>0 \mbox{ such that }  w(x) >0 \mbox{ for } |x|<b, \mbox{ and } w(x) =0 \mbox{ for } |x|\geq b \\
 \mbox{both } t_*(b) \mbox{ and } t_*(-b) \mbox{ are finite}
 \end{array}
 \right.
\right\}.
\end{equation}
In this case, the solution of \eqref{E} has exactly one right free boundary $r(t)$ and one left boundary $l(t)$. After their respective waiting times, $r(t)$ moves to the right and $l(t)$ moves to the left. One of our main purposes in this paper is to derive Darcy's law and show the strict positivity/negativity
of the velocity for each free boundary after the waiting time. That is the reason we consider such simpler initial data.
Simple examples include $w(x) = (1-x^2)_+$.

The following result is our main theorem in this section.

\begin{thm}\label{thm:Darcy0}
Assume \eqref{ass-A} and \eqref{F}. Assume further that \eqref{ass-H} holds for $k\geq 2$. Let $v(x,t)$ be the solution of \eqref{p-v} with $v_0\in \X_0$, $v_0\in C^2$ a.e. in $[-b,b]$, and $v_0$ satisfying \eqref{ass-v0xx}. Then its two free boundaries satisfy the following Darcy's law:
\begin{equation}\label{Darcy-law}
r'(t) = - v_x  (r(t)-0,t) >0 \mbox{\ for\ } t>t_*(b),\qquad l'(t) = - v_x (l(t)+0,t)<0\mbox{\ for\ } t>t_*(-b).
\end{equation}

Moreover, $r'(t)$ (resp. $l'(t)$) is continuous in $t\in (t_*(b),\infty)$ (resp. in $t\in (t_*(-b), \infty)$).
\end{thm}

\noindent
\begin{remark}\rm
This theorem consists of two parts. The first part gives Darcy's law, as in PME with a reaction (cf. \cite{LouZhou}) or without (cf. \cite{Vaz-book}). The second part shows that once a right free boundary begins to move after the waiting time, it will have a strictly positive velocity: $r'(t)>0$.
This is well known for PME without a reaction term (cf. \cite{Vaz-book}), and has partly proved for
PMEs with reactions. For example, in \cite{LouZhou} the authors proved a weaker conclusion $r'(t)\geq 0$ for PMEs with reactions. Our theorem shows that the latter conclusion can actually be improved to $r'(t)>0$, even for general reaction-diffusion equations with degenerate diffusion like \eqref{E}.

It should be noted that the positivity of the velocity is not merely a sharpening of the result; more importantly, when comparing the intersection behavior of two solutions at the free boundary by using the zero number argument, this positivity is indispensable (cf. \cite{LouLu, LouZhou}).
\end{remark}

\medskip
\noindent
{\it Proof of Theorem \ref{thm:Darcy0}}.\ \
We only consider the right free boundary $r(t)$ in $[0,T]$ for any given $T>t_*(b)$. The left free boundary is treated similarly.

Using an argument similar to that in \cite[Lemma 2.6 and Theorem 2.7]{LouZhou} for PME with reactions (see also \cite[Theorem 15.19]{Vaz-book} for PME without reactions) one can show that, for any $t\in (t_*(b), T)$,

\begin{itemize}
\item[(a).] $v_x(r(t)-0,t)$ exists, and it is the limit of $v_x(x-0,t)$ as $x\to r(t)-0$;
\item[(b).] Darcy's law holds: $r'(t) = -v_x(r(t) -0, t)$, whether $v_x(r(t) -0, t)<0$ or $v_x(r(t) -0, t)=0$;
\item[(c).] $r'(t)$ is not identically zero on any subinterval of $(t_*(b), T)$ of positive length;
\item[(d).] $r'(t)$ is continuous in $(t_*(b), T)$.
\end{itemize}

By these conclusions, it is easily seen that: there is a time sequence $\{s_n\}$ with $s_n \searrow t_*(b)$ such that $v_x(r(s_n)-0,s_n) <0$. Hence, in order to prove our conclusion, we only need to prove that: once $r'(t_0) = -v_x(r(t_0)-0,t_0)>0$ for some $t_0 \in (t_*(b), T)$, there must hold
$r'(t) = -v_x(r(t)-0,t)>0$ for all $t\in (t_0, T)$.

We prove the conclusion by contradiction. Assume that there exists $t_1\in (t_0, T)$ such that
\begin{equation}\label{construct-ass}
v_x(r(t)-0,t)<0 \mbox{ for } t\in [t_0, t_1)\quad \mbox{ and }\quad  v_x(r(t_1)-0,t_1)=0.
\end{equation}
Using the Aronson-B\'{e}nilan estimate in Proposition \ref{prop:AB}, one can construct a subsolution of ZKB type whose free boundary has strictly positive velocity and curvature bounded from below.
%More specifically, define
%\begin{equation}\label{ZKB-sol-exp}
%\underline{v}(x, t) := e^{-L_T (t-s)} \left(\frac{1-e^{-L_T (t-s)}}{L_T }\right)^{\frac{-\lambda_T}{\lambda_T+2}} \left[p-  \frac{1}{2(\lambda_T+2)}\left(\frac{1-e^{-L_T (t-s)}}{L_T}\right)^{-\frac{2}{\lambda_T+2}}(x- \tilde{x})^2 \right]_+,
%\end{equation}
%and denote the right boundary of its support by
%$$
%\underline{r}(t) := \tilde{x} + \sqrt{2(\lambda_T +2) p}
%\left( \frac{1-e^{-L_T (t-s)}}{L_T}\right)^{\frac{1}{\lambda_T+2}},
%$$
%where $\lambda_T$ and $L_T$ are the Lipschitz constants of $B(\cdot)$ and $h(x,\cdot)$, respectively, $p$, $s$ and $\tilde{x}$ are constants determined as the following.
To this end, let
$$
\lambda_T := \inf_{0<v\le V_T} \frac{B(v)}{v}>0,\qquad
L_T := \sup_{D_T} \frac{|h(x,v)|}{v}<\infty,
$$
where $D_T=\mathbb{R}\times[0,V_T]$.  The strict positivity of $\lambda_T$ follows from $B'(0)>0$ and $B(v)>0$ for $v>0$, while the finiteness of $L_T$ is a consequence of Lemma~\ref{lem:smooth-Ah-2}.
Define
\begin{equation}\label{ZKB-sol-exp}
\underline{v}(x, t) := e^{-L_T (t-s)} \left(\frac{1-e^{-L_T (t-s)}}{L_T }\right)^{\frac{-\lambda_T}{\lambda_T+2}} \left[p-  \frac{1}{2(\lambda_T+2)}\left(\frac{1-e^{-L_T (t-s)}}{L_T}\right)^{-\frac{2}{\lambda_T+2}}(x- \tilde{x})^2 \right]_+,
\end{equation}
and denote the right boundary of its support by
$$
\underline{r}(t) := \tilde{x} + \sqrt{2(\lambda_T +2)p}\,
\Bigl( \frac{1-e^{-L_T (t-s)}}{L_T}\Bigr)^{\frac{1}{\lambda_T+2}}.
$$
where $p$, $s$ and $\tilde{x}$ are constants determined as follows.

By Proposition \ref{prop:AB} we have
$$
v_{xx}(x,t)\geq -C_*,\qquad x\in \R,\ t\in [t_0, t_1]
$$
for some $C_*> 0$. Denote
$$
\hat{a} := \frac{(\lambda_T+1)C_*}{(\lambda_T+1)C_*+L_T},\qquad s(t_0):= t_0 +\frac{\log \hat{a}}{L_T}.
$$
So $\hat{a}\in (0,1)$ and $s(t_0)<t_0$.  Taking $s=s(t_0)$ in $\underline{v}$, then we have
$$
\underline{v}_{xx}(x, t_0)= -\frac{L_T e^{-L_T (t_0-s(t_0))}}{(\lambda_T+1)[1- e^{-L_T (t_0-s(t_0))}]}= -C_*,
$$
and
$$
\underline{r}'(t_0)= \frac{\sqrt{2(\lambda_T+1)p}}{\lambda_T+2}\left[\frac{1-e^{-L_T (t_0-s(t_0))}}{L_T}\right]^{-\frac{\lambda_T+1}{\lambda_T+2}} e^{-L_T (t_0-s(t_0))}.
$$
Moreover, we use the equality $\underline{r}'(t_0)=k_0 := -v_x(x_0-0,t_0) = r'(t_0)>0$ to determine $p=p(t_0)$, and use the equality $\underline{r}(t_0)= x_0:= r(t_0)$ to determine
$$
\tilde{x} := x_0- \sqrt{2(\lambda_T+1)p(t_0)}\left[\frac{1-e^{-L_T (t_0- s(t_0))}}{L_T}\right]^{\frac{1}{\lambda_T+2}}.
$$

Choosing such $s,p,\tilde{x}$, we conclude that $\underline{v}(x, t_0)\leq v(x, t_0)$, and
$$
\underline{v}_t=\lambda_T \underline{v} \underline{v}_{xx}+ \underline{v}_x^2- L_T \underline{v} \leq B(\underline{v})\underline{v}_{xx}+ \underline{v}_x^2+ h(x, \underline{v}).
$$
Here we use the fact $\underline{v}_{xx}\leq 0$ in its support.  By comparison we have
$$
\underline{v}(x, t)\leq v(x, t),\quad \mbox{ for } t\in J:= [t_0, t_1].
$$
Hence, for any small $\tau >0$, we have
\begin{equation}\label{under-r<r}
\underline{r}(t_0+ \tau)- \underline{r}(t_0)- \tau \underline{r}'(t_0)\leq r(t_0+ \tau)- r(t_0)- \tau r'(t_0).
\end{equation}
By a direct calculation we have $\underline{r}''(t_0)=- \sigma \underline{r}'(t_0)$
for some $\sigma = \sigma(L_T, \lambda_T, \hat{a})>0$, and so
\begin{equation}\label{estimate-r''}
\frac{r(t_0+ \tau)- r(t_0)- \tau r'(t_0)}{\tau^2}\geq -\frac{\sigma}{2}\underline{r}'(t_0)- C_1 \tau =
-\frac{\sigma}{2} r'(t_0)- C_1 \tau,
\end{equation}
for some $C_1>0$ independent of $t_0$.

In a similar way, one can show that \eqref{estimate-r''} actually holds for all $t\in [t_0, t_1)$ instead of only at $t_0$, with the same constant $\sigma>0$.

Now, for any small $\tau_0>0$ and any nonnegative test functions $\rho \in C_0^{\infty}([t_0+\tau_0, t_1])$, we show that
\begin{equation}\label{r''>0}
-\int_{t_0}^{t_1} r'(t)\rho'(t) dt= \int_{t_0}^{t_1} r(t) \rho''(t) dt\geq -\int_{t_0}^{t_1} \sigma r'(t)\rho(t) dt.
\end{equation}
In fact, since \eqref{estimate-r''} holds when $t_0$ is replaced by $t\in [t_0, t_1)$, for any $0\leq \tau \leq \tau_0$ we have
\begin{equation}\label{I-tau>}
I(\tau) := \int_{t_0}^{t_1} \frac{r(t+ \tau)- r(t)- \tau r'(t)}{\tau^2} \rho(t) dt \geq - \int_{t_0}^{t_1} \left[ \frac{\sigma}{2} r'(t)+  C_1 \tau\right] \rho(t)dt.
\end{equation}
Extending both $\rho(t)$ and $r(t)$ to $[t_0, t_1+\tau]$ by zero and using the fact
$$
\rho(s-\tau) -\rho(s) = -\tau \rho'(s) +\frac{\tau^2}{2} \rho''(s) + o(\tau^2),
$$
we obtain
\begin{eqnarray*}
I(\tau) & = & \frac{1}{\tau^2}\int_{t_0+ \tau}^{t_1+ \tau} [\rho(s-\tau)- \rho(s)]r(s) ds- \frac{1}{\tau}\int_{t_0}^{t_1} \rho(t)r'(t) dt \\
& =  & -\frac{1}{\tau}\int_{t_0+ \tau}^{t_1+ \tau} \rho'(s)r(s)ds+ \frac12 \int_{t_0+ \tau}^{t_1+ \tau} \rho''(s)r(s) ds- \frac{1}{\tau}\int_{t_0}^{t_1} \rho(t)r'(t) dt+ o(1) \\
& = & \frac12 \int_{t_0}^{t_1} \rho''(t)r(t) dt+ o(1).
\end{eqnarray*}
Combining with \eqref{I-tau>} we have
$$
\int_{t_0}^{t_1} \rho''(t)r(t) dt+ o(1)  \geq  -\int_{t_0}^{t_1} \left[\sigma r'(t) + 2C_1 \tau\right] \rho(t) dt,
$$
which implies \eqref{r''>0} by letting $\tau\to 0$. Furthermore, by letting $\tau_0\to 0$ we see that \eqref{r''>0} holds actually for all nonnegative $\rho\in C_0^\infty ([t_0,t_1])$.

By \eqref{r''>0} we have
$$
\int_{t_0}^{t_1} (r'(t) e^{\sigma t})' \rho(t) dt = \int_{t_0}^{t_1} [r''(t) +\sigma r(t)]e^{\sigma t}\rho (t) dt \geq 0.
$$
Consequently, the quantity $ e^{\sigma t}r'(t)$ is nondecreasing on \([t_0,t_1] \).
This, however, contradicts the fact \(r'(t_0)>0\) and the assumption \( r'(t_1)=0.\)

This completes the proof of Theorem \ref{thm:Darcy0}.
\qed

\section{Higher Regularities}
In this section we study the higher regularities of $v$ and its free boundaries.
%Combined with the Darcy law $r'(t)=-v_x(r(t)-0,t)$ (for $t>t_*(b)$), the regularity of the free boundary reduces to the regularity of derivatives of $v$ up to the interface.
Throughout this section, we still consider initial data chosen from $\X_0$, and so $v(\cdot,t)$ is positive if and only if $x\in (l(t),r(t))$.

\subsection{Upper Bound of $v_{xx}$}
In Section 2 we have given a lower bound for $v_{xx}$, which is used in the approach in Section 3. Now we give
a local boundary estimate for $v_{xx}$.

\begin{lem}\label{lem:vxx-bound}
Assume \eqref{ass-A} and \eqref{F}. Assume further that \eqref{ass-H} holds for $k\geq 2$. Let \(v \) be the solution of \eqref{p-v} with initial data $v_0 \in \X_0$. Then, for any $t_0> t_*(b)$, there exists a neighbourhood $N(P_0)\subset Q_T $ of $P_0 := (r(t_0), t_0)$ and a constant $C'_2 = C'_2 (P_0)>0$, such that
\[
|v_{xx}(x,t)|\le C'_2 \qquad\text{in }N(P_0).
\]
\end{lem}

\begin{proof}
We first discuss the lower bound. From the Aronson-B\'{e}nilan estimate in Proposition \ref{prop:AB} we already have a lower bound for $v_{xx}$, but it depends on the initial data $v_0(x)$.

We now use a Bernstein-type argument to obtain another lower bound that holds away from $t=0$ and requires no additional condition on $v_0$.
More precisely, we take $g_0 := \frac13 (t_0 - t_*(b))$ and choose a function
$$
\rho(t) :=\left\{
\begin{array}{ll}
0, & t\in [0, t_*(b) + g_0],\\
\mbox{smooth and increasing}, & t\in [t_0 -2g_0, t_0 -g_0],\\
1, & t\geq t_0 - g_0,
\end{array}
\right.
$$
such that whenever $|\frac{\rho_t}{\rho}|\leq \Theta$ and
$|\rho^{(j)} (t)| \leq \theta_j$ for all $j\geq 1,\ t\in [0,\infty)$ (when $\rho=0$ we interpret $\frac{\rho_t}{\rho}$ as $0$). Then we consider the function $\tilde{\eta}:= \rho(t) \eta(x,t)$. By the equation \eqref{equ-eta} we have
\begin{equation}\label{eq-tilde-eta}
\tilde{\eta}_t = a \tilde{\eta}_{xx} + b \tilde{\eta}_x + \Big(c +\frac{\rho_t}{\rho}\Big) \tilde{\eta} + d\frac{\tilde{\eta}^2}{\rho} + e \rho ,\quad x\in \R, \ t>0,
\end{equation}
where $\frac{\rho_t}{\rho}$ and $\frac{\tilde{\eta}}{\rho}$ are regarded as $0$ and $\eta$, respectively, when $t\in [0,t_*(b)+g_0]$.  Then $-D(t+\tau)$ for suitable  positive $D$ and $\tau$ (which may be different from those in Proposition \ref{prop:AB}) is a subsolution to \eqref{eq-tilde-eta} and so
$$
\tilde{\eta}(x,t) \geq -D(t+\tau),\qquad x\in \R, \ t\geq 0,
$$
which gives the lower bound for $\eta$ when $t\geq t_0 -g_0$, and the bound is independent of $v_0(x)$.

\medskip

Now we establish the upper bound. Differentiating $v(r(t),t)=0$ and using
Darcy's law we have $v_t(r(t),t) = v_x^2 (r(t)-0,t)$ for $t>t_*(b)$. Substituting it into the equation of $v$ we have
\begin{equation}\label{Bvxx}
B(v)v_{xx}\to0 \qquad  \text{as } (x,t)\to(r(t)-0,t),
\end{equation}
We will show the conclusion by this limit together with
$$
v(x,t)\sim -v_x(r(t)-0,t)(r(t)-x).
$$

First we recall that the function $\eta := v_{xx}$ satisfies
\begin{align*}
\mathcal{N}(\eta) & := \eta_t-  B(v) \eta_{xx} - 2\big[ B'(v) + 1 \big] v_x \eta_x - \big[ B''(v) v_x^2 + h_v \big] \eta \\
&\quad - \big[B'(v) + 2\big] \eta^2 - \big[ h_{xx} + 2 h_{xv} v_x + h_{vv} v_x^2\big]= 0.
\end{align*}
By our assumption \eqref{ass-H} for $k\geq 2$,
$$
|B''(v)v_x^2 + h_v|+ |B'(v) + 2|+ |h_{xx} + 2h_{xv}v_x + h_{vv}v_x^2|\leq E_0,
$$
for some $E_0>0$.  Denote $k_0 := -v_x(x_0-0,t_0)>0$, where $x_0 := r(t_0)$. For any given $\varepsilon\in (0, k_0)$, we can find $\delta_1 = \delta_1 (\varepsilon)> 0$ and $\tau_1 = \tau_1 (\varepsilon)\in (0, t_0 -t_*(b))$ such that for all
$$
(x, t)\in N_{\delta_1, \tau_1}(P_0) :=  \{(x, t)|r(t)-\delta_1 \leq x< r(t),\ t_0-\tau_1 \leq t\leq t_0+ \tau_1 \},
$$
there hold
\begin{equation}\label{neigh}
\left\{
\begin{array}{ll}
(k_0- \varepsilon)(r(t)- x)\leq v(x, t)\leq (k_0+ \varepsilon)(r(t)- x),\\
k_0- \varepsilon\leq r'(t)\leq k_0+ \varepsilon,\\
|B(v)v_{xx}(x, t)|\leq \varepsilon,\\
\left|\frac{B(v)}{v}- B'(v)\right|\leq \left|\frac{B(v)}{v}- \OM \right|+ \left|B'(v)- \OM \right|< 2\varepsilon,
\end{array}
\right.
\end{equation}
where $\OM= B'(0^+)$.   We take the following parameters
$$
\varepsilon := \min\left\{\frac12, \frac{\OM}{2}, \frac{k_0}{16(4\OM+ 6)}, \frac{k_0^2 \OM}{k_0\OM+ 64E_0}\right\},\quad \delta := \min\left\{\delta_1(\varepsilon), \frac{k_0}{32E_0}\right\},\quad \tau := \min\left\{\tau_1 (\varepsilon), \frac{\delta}{6\varepsilon}\right\}.
$$

Define $r^*(t) := r(t_0-\tau )+ (k_0+ 2\varepsilon)(t- t_0+ \tau )$. Then one can construct a barrier of the form
$$
\bar\eta(x, t) := \frac{\alpha}{r(t)-x} + \frac{\beta}{r^*(t)-x},\quad (x,t)\in N_{\delta, \tau}(P_0),
$$
for some $\alpha, \beta>0$ to be determined. Then, in the domain $N_{\delta, \tau}(P_0)$ (note that $r(t)-x\leq \delta$ in this domain), a direct calculation yields
\begin{align*}
\mathcal{N}(\bar{\eta}) &\geq \frac{\alpha}{(r(t) - x)^2} \left\{ -(k_0+ \varepsilon) - \frac{2B(v)}{v}(k_0+ \varepsilon) + 2[B'(v) + 1](k_0- \varepsilon) - E_0\delta- 2E_0\alpha \right\} \\
&\quad + \frac{\beta}{(r^*(t) - x)^2} \bigg\{ -(k_0 + 2\varepsilon) - \frac{2B(v)}{v}\frac{(k_0+ \varepsilon)(r(t)-x)}{r^*(t) - x} + 2[B'(v) + 1](k_0- \varepsilon) \\
&\quad - E_0(\delta+ 6\varepsilon \tau )- 2E_0\beta \bigg\} - E_0 \geq 0,
\end{align*}
provided $\beta := \frac{k_0}{16 E_0}> 0$ and $\alpha\in (0, \frac{k_0}{8E_0}]$.

On the other hand, at the bottom of $N_{\delta, \tau}(P_0)$ we have
$$
\bar{\eta}(x,t_0-\tau)> \frac{\beta}{r(t_0- \tau)-x}\geq \frac{k_0}{32 E_0 [r(t_0 - \tau)- x]}> \frac{2\varepsilon}{(k_0- \varepsilon)\OM[r(t_0 - \tau)- x]}\geq \eta(x,t_0-\tau).
$$
On the left boundary of $N_{\delta, \tau}(P_0)$ we have
$$
\bar{\eta}(r(t)- \delta,t)= \frac{\alpha}{\delta}+\frac{\beta}{r^*(t)-r(t)-\delta}> \frac{\beta}{6\varepsilon\tau}> \frac{k_0}{32E_0\delta}\geq \frac{\varepsilon}{(k_0- \varepsilon)\delta}\frac{2}{\OM}\geq \frac{\varepsilon}{v}\frac{v}{B(v)}\geq \eta(r(t)- \delta,t).
$$
Finally, we deal with the right boundary of $N_{\delta, \tau}(P_0)$. More precisely, we compare $\eta$ and $\bar{\eta}$ on a narrow strip inside the right boundary. For any given $\alpha\in(0, \frac{k_0}{8E_0}]$, there exists $\tilde{\delta}(\alpha)>0$ sufficiently small such that
$$
B(v)v_{xx}\leq \alpha(k_0- \varepsilon)\frac{\OM}{2} \quad \mbox{ for }\ r(t)-\tilde{\delta} \leq x< r(t).
$$
Then
$$
\eta= v_{xx}\leq \frac{\alpha(k_0- \varepsilon)}{v}\frac{\OM}{2}\frac{v}{B(v)}\leq \frac{\alpha}{r(t)-x}< \bar{\eta} \quad \mbox{ in } N_{\tilde{\delta}, \tau}(P_0).
$$
By the comparison principle we conclude that
$$
\eta(x,t) = v_{xx}(x, t)\leq \bar{\eta}(x, t)= \frac{\alpha}{r(t)- x}+ \frac{\beta}{r^*(t)- x} \quad \mbox{ in } N_{\delta, \tau}(P_0),
$$
for every $\alpha\in(0, \frac{k_0}{8E_0}]$ and $\beta= \frac{k_0}{16 E_0}$.
Noting $r^*(t)-x \geq \tau$ in  $N_{\delta, \frac{\tau}{2}} (P_0)$, and letting $\alpha\downarrow 0$ gives the desired local upper bound
$C'_2 =C'_2 (k_0, E_0, \delta_1, \tau_1) = C'_2 (P_0)$ in the neighborhood $N_{\delta, \frac{\tau}{2}} (P_0)$ of $P_0$.
\end{proof}

Combining this local boundary estimate with the interior estimate we have the following local bound for $v_{xx}$.

\begin{prop}\label{prop:vxx-whole}
Assume \eqref{ass-A} and \eqref{F}. Assume further that \eqref{ass-H} holds for $k\geq 2$. Let \(v \) be the solution of \eqref{p-v} with initial data $v_0 \in \X_0$. For any $T_2 >T_1 >\max\{t_*(b),t_*(-b)\}$, there exists $C_2$ depending on $T_1, T_2$ and the constants in \eqref{ass-A}, \eqref{F} and \eqref{ass-H} such that
$$
|v_{xx}(x,t)|\leq C_2 \mbox{\ in \ } D(T_1,T_2):= \{(x,t) \mid l(t)<x<r(t),\ t\in [T_1, T_2]\}.
$$
\end{prop}

\subsection{The Bound of $v_{xxx}$}
We continue to establish the bound of $v_{xxx}$. Our main result is as the following.

\begin{prop}\label{prop:vxxx}
Suppose all conditions of Proposition \ref{prop:vxx-whole} hold, with the condition \eqref{ass-H} holds for $k\geq 2$ replaced by \eqref{ass-H} holds for $k\geq 3$.
Then, for any $T_2 >T_1 >\max\{t_*(b),t_*(-b)\}$, there exists $C_3$ depending on $T_1, T_2$ and the constants in \eqref{ass-A}, \eqref{F} and \eqref{ass-H} such that
$$
|v_{xxx}(x,t)|\leq C_3 \mbox{\ in \ } D(T_1,T_2):= \{(x,t) \mid l(t)<x<r(t),\ t\in [T_1, T_2]\}.
$$
\end{prop}

\medskip
Denote $\xi (x,t):= v_{xxx}(x,t)$.
Differentiating the pressure equation yields
\begin{equation}\label{vxxx-equ}
\mathcal{N}_3 (\xi) := \xi_t- B(v)\,\xi_{xx} - \bigl[3B'(v)+2\bigr]v_x\,\xi_x - c_3(x, t)\xi -e_3(x, t)= 0,
\end{equation}
where
$$
c_3(x, t) := 3B''(v)v_x^2 + \bigl(4B'(v)+6\bigr)v_{xx} + h_v(x,v)
$$
and
$$
\begin{aligned}
e_3(x, t) & :=  B'''(v)v_x^3v_{xx} + 3B''(v)v_x v_{xx}^2 \\
&\quad + h_{xxx} + 3h_{xxv}v_x + 3h_{xvv}v_x^2 + h_{vvv}v_x^3 + 3h_{xv}v_{xx} + 3h_{vv}v_x v_{xx}.
\end{aligned}
$$
If we further assume $B'''(v)$ and all third-order partial derivatives of $h(x, v)$ are bounded for $(x, v)\in \R\times[0, V_T]$, then we can establish the lower bound
$$
\xi(x, t)\geq -D_3(t+ t_0), \quad x\in\R,\ t\in(0, T],
$$
for sufficiently large constant $D_3$ by repeating the process in \ref{prop:AB}.
We want to construct a barrier function as for $\eta = v_{xx}$ in Lemma \ref{lem:vxx-bound}.
The crucial difference is that the estimates \eqref{Bvxx} for $\eta$ near $x=r(t)$ can proceed directly by the structure of the equation, whereas those for $\xi$ necessitate a more intricate near-boundary estimate. For this purpose we first prepare a lemma.

\begin{lem}\label{lem:vxxx-est}
Assume the hypotheses of Proposition \ref{prop:vxxx} hold. Let $P_0$, $\delta$ and $\tau$ be the same as in
Lemma \ref{lem:vxx-bound}. Then, there exist constants $C'_3 > 0$, $\delta'\in(0, \delta)$ and $\tau'\in (0, \frac{\tau}{2})$, depending on $k_0$ and $C'_2(P_0)$ such that
$$
|\xi(x, t)|\leq \frac{C'_3}{r(t)-x}  \quad \mbox{ in } N_{\delta', \tau'}(P_0).
$$
\end{lem}

\begin{proof}
Fix a small $\varepsilon\in (0, k_0)$. We take
$$
\delta' := \min\left\{\frac23 \delta,\ (k_0- \varepsilon)\tau\right\},\quad \tau' := \frac{\tau}{2}-\frac{\delta}{4(k_0- \varepsilon)} \quad \mbox{and}\quad (\bar{x}, \bar{t})\in N_{\delta', \tau'}(P_0)
$$
and define
$$
\mathcal{R}(\bar{x}, \bar{t}):=\left\{(x, t)\left||x-\bar{x}|\leq \frac{\lambda}{2},\ \bar{t}-\frac{\lambda}{4(k_0-\varepsilon)}\leq t\leq \bar{t}\right.\right\},
$$
where $\lambda := r(\bar{t})-\bar{x}$. Then, on the right boundary of this rectangle we have
\begin{equation}\label{r(t)>x+lambda}
x  =  \bar{x}+ \frac{\lambda}{2} = r(\bar{t}) - \frac{\lambda}{2} <  r(\bar{t}) - \frac{\lambda}{4} \leq r(\bar{t}) + (k_0 -\varepsilon) (t-\bar{t}) \leq r(t)
\end{equation}
since $\bar{t}-\frac{\lambda}{4(k_0-\varepsilon)}\leq t\leq \bar{t}$. Here we use the conditions in \eqref{neigh}. This implies that $\mathcal{R}(\bar{x}, \bar{t}) \subset N_{\delta, \frac{\tau}{2}}(P_0)$.

Now we introduce  changes of variables involving translation and scaling:
$$
y := \frac{x-\bar{x}}{\lambda},\quad s :=\frac{t-\bar{t}}{\lambda},
$$
and define
$$
\zeta(y, s) := \eta(\lambda y + \bar{x}, \lambda s + \bar{t})= \eta(x, t),
$$
which satisfies
\begin{equation}\label{zeta-equ}
\begin{aligned}
\zeta_s &= \frac{B(v)}{v}\frac{v}{\lambda} \zeta_{yy} + 2\big[ B'(v) + 1 \big] v_x \zeta_y + \lambda\big[ B''(v) v_x^2 + h_v+ \big(B'(v) + 2\big)v_{xx} \big] \zeta \\
&\quad  + \lambda\big[ h_{xx} + 2 h_{xv} v_x + h_{vv} v_x^2\big], \quad -\frac12\leq y \leq \frac12,\ -\frac{1}{4(k_0- \varepsilon)}\leq s\leq 0.
\end{aligned}
\end{equation}
By Lemma \ref{lem:vxx-bound} we have $|\zeta|\leq C'_2(P_0)$ for $-\frac12\leq y \leq \frac12$ and $-\frac{1}{4(k_0- \varepsilon)}\leq s\leq 0$.
We explain that the equation of $\zeta$ is a uniformly parabolic one. In fact,
by \eqref{neigh} we have
$$
\frac{k_0 + \varepsilon}{4}\leq (k_0-\varepsilon)\frac{r(t)-x}{\lambda}\leq \frac{v(x,t)}{\lambda}\leq (k_0+\varepsilon)\frac{r(t)-x}{\lambda} \leq \frac{3(k_0+\varepsilon)}{2}\quad \mbox{ in }N_{\delta, \frac{\tau}{2}}(P_0).
$$
The first inequality follows from \eqref{r(t)>x+lambda}. This implies
$$
\frac{\OM}{2}\frac{k_0-\varepsilon}{4}\leq\frac{B(v)}{v}\frac{v}{\lambda}\leq \frac{3\OM}{2}\frac{3(k_0+ \varepsilon)}{2}.
$$
So the equation \eqref{zeta-equ} is uniformly parabolic in $[-\frac12,  \frac12]\times [-\frac{1}{4(k_0- \varepsilon)}, 0]$. By using the Schauder interior estimates  we have
$$
|\zeta_y(0, 0)|\leq C'_3,
$$
for some $C'_3>0$ depending on $\varepsilon$, $k_0$ and $C'_2(P_0)$.
In particular, when we take $\varepsilon =\frac{k_0}{2}$, $C'_3$ can be chosen depending on $k_0$ and $C'_2(P_0)$.
Consequently,
$$
|\xi(\bar{x}, \bar{t})|=|\eta_x(\bar{x}, \bar{t})|\leq \frac{C'_3}{\lambda}= \frac{C'_3}{r(\bar{t})-\bar{x}}.
$$
By the arbitrariness of $(\bar{x}, \bar{t})$, we complete the proof of the lemma.
\end{proof}

Based on this lemma, we can prove Proposition \ref{prop:vxxx}.

\medskip
\noindent
{\it Proof of Proposition \ref{prop:vxxx}.}\ \
As the estimate for $v_{xx}$ in Lemma \ref{lem:vxx-bound}, we want to construct a barrier function of the form
$$
\bar{\xi}(x, t) := \frac{\alpha_m}{r(t)- x- \frac{1}{3n}}+ \frac{\beta_m}{r^*(t)-x},\quad  (x, t)\in N'_{\delta', \tau'}(P_0),
$$
where $\frac1n< \delta'$, the parameters $\alpha_m>0$ and $\beta_{m}>\frac{C'_3}{3}$ are two sequences to be determined later, the domain
$$
N'_{\delta', \tau'}(P_0) :=\left\{(x, t)\ \left|\ r(t)-\delta'\leq x< r(t)-\frac1n,\  t_0-\tau'\leq t\leq t_0+ \tau' \right. \right\},
$$
and $r^*(t) := r(t_0- \tau')+ (k_0+ 2\varepsilon)(t- t_0+ \tau')$.
By a direct calculation, we have
\begin{align*}
\mathcal{N}_3 (\bar{\xi}) &= \frac{\alpha_m}{(r(t) - x - \frac{1}{3n})^2} \left\{ -r'(t) - \frac{2B(v)}{r(t) - x - \frac{1}{3n}} - [3B'(v) + 2]v_x - c_3(x, t)\left(r(t) - x- \frac{1}{3n}\right) \right\} \\
&\quad + \frac{\beta_m}{(r^*(t) - x)^2} \left\{ -r'^{*}(t) - \frac{2B(v)}{r^*(t) - x} - [3B'(v) + 2]v_x - c_3(x, t)(r^*(t) - x) \right\}- e_3(x, t).
\end{align*}
Note that
\begin{align*}
&\left(r(t)-\frac{1}{3n}-x\right)(r(t)-x)- \left(r(t)-\frac{1}{n}-x\right)\left(r(t)-\frac{1}{3n}-x\right) \\
> &\left(r(t)-\frac{1}{3n}-x\right)(r(t)-x)- \left(r(t)-\frac{1}{n}-x\right)(r(t)-x)\mbox{\ \ in\ } N'_{\delta', \tau'}(P_0).
\end{align*}
So
$$
\frac{r(t)- x}{r(t)-x-\frac{1}{3n}}< \frac32\mbox{\ \ in\ } N'_{\delta', \tau'}(P_0).
$$
and
$$
\frac{v}{r^*(t)- x}\leq \frac{v}{r(t)- x- \frac{1}{3n}}< \frac{(k_0+ \varepsilon)(r(t)-x)}{r(t)- x- \frac{1}{3n}}<\frac32(k_0+\varepsilon)\mbox{\ \ in\ } N'_{\delta', \tau'}(P_0),
$$
Comparing with the equation of $\eta$, we find that the equation of $\xi$ does not contain the $\xi^2$ term. Therefore, we do not need to impose additional conditions on the coefficients $\alpha_m$ and $\beta_m$ to ensure $\mathcal{N}_3(\bar{\xi})\geq 0$ while it suffices to take the parameters $\delta'$ and $\tau'$ sufficiently small.

Next we compare $\bar{\xi}$ and $\xi$ on the parabolic boundary of $N'_{\delta', \tau'}(P_0)$. We first set $\alpha_1 := \frac{2C'_3}{3}$ and $\beta_1 := \frac{2C'_3}{3}$, then at $t = t'_1 := t_0- \tau'$, we have
$$
\bar{\xi}(x, t'_1)= \frac{2C'_3}{3(r(t'_1)-x- \frac{1}{3n})}+ \frac{2C'_3}{3(r(t'_1)-x)}> \frac{4C'_3}{3(r(t'_1)-x)}> \xi(x, t'_1).
$$
By \eqref{neigh} and the definition of $r^*(t)$, we have $0< r^*- x< \delta' + 6\varepsilon \tau'< 2\delta'$ (by taking $\tau'< \frac{\delta'}{6\varepsilon}$). So,  at $x= r(t)- \delta'$, we have
$$
\bar{\xi}(r(t)- \delta', t)\geq \frac{2C'_3}{3(\delta'- \frac{1}{3n})}+ \frac{C'_3}{3 \delta'}\geq \frac{C'_3}{\delta'}\geq \xi(x- \delta', t).
$$
At $x= r(t)- \frac1n$, we also have
$$
\bar{\xi}\left(r(t)- \frac1n, t \right)= \frac{2C'_3}{3(\frac1n- \frac{1}{3n})}+ \frac{2C'_3}{3(r^*(t)- r(t)+ \frac1n)}\geq n C'_3 \geq \xi \left(x- \frac1n, t \right).
$$
Consequently, by comparison principle we can get
$$
\xi(x, t)\leq \frac{2C'_3}{3(r(t)- x- \frac{1}{3n})}+ \frac{2C'_3}{3(r^*(t)-x)} \quad \mbox{in} \ N'_{\delta', \tau'}(P_0).
$$

Next, we take $\alpha_2 := \left(\frac{2}{3}\right)^2C'_3$ and $\beta_2 := \frac{2C'_3}{3}+ \left(\frac{2}{3}\right)^2C'_3$. Based on the preceding discussion, $\mathcal{N}_3 (\bar{\xi})\geq 0$ still holds. In addition, at $t= t'_1$, we have
$\bar{\xi}(x, t'_1) > \xi(x, t'_1)$; At $x= r(t)- \delta'$ and at $x=r(t)-\frac1n$,
using the estimate for $\alpha_1$ and $\beta_1$ we have $\bar{\xi}(x, t)\geq \xi(x, t)$ for $t\in [t_0-\tau', t_0+\tau']$. Then by comparison we have
$$
\xi(x, t)\leq \frac{\left(\frac{2}{3}\right)^2C'_3}{r(t)- x- \frac{1}{3n}}+ \frac{\frac{2C'_3}{3}+ \left(\frac{2}{3}\right)^2C'_3}{r^*(t)-x} \quad \mbox{in} \ N'_{\delta', \tau'}(P_0).
$$
By iteration, when we take $\alpha_m := \left(\frac{2}{3}\right)^mC'_3$ and $\beta_m := \displaystyle\sum_{i=1}^m \left(\frac23\right)^i C'_3= 2\left[1- \left(\frac23\right)^m\right] C'_3$, we have
$$
\xi(x, t)\leq \frac{\left(\frac{2}{3}\right)^mC'_3}{r(t)- x- \frac{1}{3n}}+ \frac{2\left[1- \left(\frac23\right)^m\right] C'_3}{r^*(t)-x} \quad \mbox{in} \ N'_{\delta', \tau'}(P_0).
$$

Finally, by letting $n\to \infty$ and $m\to \infty$, we derive that, for each $(x, t)\in N_{\delta', \frac{\tau'}{2}}(P_0)$,
$$
v_{xxx}(x, t)< \frac{2C'_3}{r^*(t)- x}\leq \frac{2C'_3}{[r(t'_1)+(k_0+ 2\varepsilon)\frac{\tau'}{2}]- [r(t'_1)+(k_0+ \varepsilon)\frac{\tau'}{2}]}= \frac{4C'_3}{\varepsilon \tau'}.
$$
This actually gives the local boundary estimates for $v_{xxx}$.

Combining the above boundary estimate and the interior one, one can derive the desired bound for $v_{xxx}$.
\qed

\subsection{The Bound of $\partial_x^k v $}

\begin{thm}\label{thm:k-th order}
Suppose all conditions of Proposition \ref{prop:vxx-whole} hold, with the condition \eqref{ass-H} holds for $k\geq 2$.
Then, for any $T_2 >T_1 >\max\{t_*(b),t_*(-b)\}$, there exists $C_k$ depending on $T_1, T_2$ and the constants in \eqref{ass-A}, \eqref{F} and \eqref{ass-H} such that
$$
|\partial^k_x v (x,t)|\leq C_k \mbox{\ in \ } D(T_1,T_2):= \{(x,t) \mid l(t)<x<r(t),\ t\in [T_1, T_2]\}.
$$
\end{thm}

\begin{proof}
The conclusion has been proved above for the case when $k=2$ and $k=3$. For general $k\geq 4$, we denote $v^{(k)} := \partial_x^k v$. Then $v^{(k)}$ satisfies the following equation:
$$
\begin{aligned}
v^{(k)}_t &= B(v)\, v^{(k+2)} + \bigl(k B'(v) + 2\bigr) v^{(1)} v^{(k+1)} \\
&\quad + \Bigg[ \frac{k(k-1)}{2} B''(v) \bigl(v^{(1)}\bigr)^2 + \left( \frac{k(k-1)}{2} B'(v) + 2k \right) v^{(2)} + h_v(x,v) \Bigg] v^{(k)} \\
&\quad + P_k\bigl(x, v, v^{(1)}, v^{(2)}, \dots, v^{(k-1)}\bigr),
\end{aligned}
$$
where $P_k$ denotes a polynomial of $v^{(1)}, \dots, v^{(k-1)}$ whose coefficients are completely determined by $B(v)$, $h(x,v)$ and their derivatives. By repeating the argument in Lemma \ref{lem:vxxx-est}, we can find a positive constant \(C'_k\) and parameters $\delta_k$ and $\tau_k$ such that the estimate $|v^{(k)}|\leq \frac{C'_k}{r(t)-x}$ holds in a neighborhood $N_{\delta_k, \tau_k}(P_0)$ of $P_0$. By using lower and upper solutions of the following form
$$
\underline{v}^{(k)}= -D_k(t+ t_k) \quad \mbox{ for some positive constant}\quad D_k, t_k,
$$
and
$$
\bar{v}^{(k)}= \frac{\left(\frac{2}{3}\right)^n C'_k}{r(t)- x- \frac{1}{3n}}+ \frac{2\left[1- \left(\frac23\right)^n\right] C'_k}{r^*(t)-x} \quad \mbox{in} \quad N_{\delta_k, \tau_k}(P_0),
$$
we can eventually obtain the local boundary estimates for $v^{(k)}$. The estimates in $D(T_1, T_2)$ follow from the proved boundary estimate and the interior one.
\end{proof}

\subsection{High Regularity on the Boundaries}
Note that the above regularity results hold in $D(T_1, T_2)$, which does not include the side boundaries
$$
S_l := \{(x,t) \mid x= l(t),\ t\in [T_1, T_2]\},
\qquad
S_r := \{(x,t) \mid x= r(t),\ t\in [T_1, T_2]\}.
$$
Hence, the existence of the $k$-th order left (resp. right) derivative of $v$:
$$
\partial_x^k v(r(t)-0,t):=  \lim\limits_{\epsilon\to 0^+}\frac{\partial_x^{k-1} v(r(t)-\epsilon,t) - \partial_x^{k-1} v(r(t)-0,t)}{-\epsilon }
$$
(resp. $\partial_x^k v(l(t)+0,t)$) on $S_r$ (resp. $S_l$),
as well as its boundedness, is yet open.
We conjecture that these properties cannot be deduced solely from the boundedness of $\partial^k_x v$ in $D(T_1, T_2)$.
Nevertheless, they are true when we require the interior regularity a little higher.

\begin{thm}\label{thm:r-smooth}
Assume \eqref{ass-A} and \eqref{F}. Assume further that \eqref{ass-H} holds for $k\geq 2$. Let \(v \) be the solution of \eqref{p-v} with initial data $v_0 \in \X_0$. Then, for any $t_0 > t_*(b)$,
\begin{enumerate}
\item[\rm (i).] $\partial_x^{k-1} v(r(t_0)-0,t_0)$ exists, and it is bounded by $C_{k-1}$ in the previous theorem;
\item[\rm (ii).] $r^{(k-1)}(t_0)$ exists and it is continuous in $t_0$.
\end{enumerate}
\end{thm}

\begin{proof}
Without loss of generality, we consider only the case where $k=3$. Other cases can be studied similarly.

For any small $\delta>0$, we denote as above
$$
D(t_0-\delta, t_0+\delta) := \{(x,t) \mid l(t)<x<r(t),\ |t-t_0|\leq \delta\}.
$$
Then, under our assumption, there exist positive constants $C^*_j\ (j=0,1,2,3)$ such that
$$
|\partial^j_x v(x,t)| \leq C^*_j, \quad (x,t)\in D(t_0-\delta, t_0+\delta), \ j=0,1,2,3.
$$
Using these bounds in the interval $(l(t),r(t))$, it is easily seen that the following limits exist:
$$
H_1(t) := \lim\limits_{x\to r(t)-0} v_x(x,t),\qquad H_2 (t):= \lim\limits_{x\to r(t)-0} v_{xx} (x,t).
$$
By the Derivative Limit Theorem we have
$$
H_1(t) = v_x(r(t)-0,t) = - r'(t),\qquad H_2(t) = v_{xx} (r(t)-0,t).
$$
This proves our conclusion: when \eqref{ass-H} holds for $k=3$, the second-order left derivative of $v$
on the right boundary is well defined, and, for each given $t$, $v_{xx}(\cdot,t)$ is continuous till the boundary.
Furthermore, since $v_{xx}(\cdot,t)$ is Lipschitz in $x\in (l(t),r(t))$ with Lipschitz constant $C^*_3$ (which is uniform
in $t\in [t_0-\delta, t_0 + \delta]$), one can show that, as a function of two variables, $v_{xx}(x,t)$ is continuous
in $\overline{D^+}$, where $D^+ := D(t_0-\delta, t_0+\delta) \cap \{x>0\}$.

Next we study $r''(t)$. Consider the equation of $v_x$:
$$
v_{tx} = B(v) v_{xxx} + [B'(v)+2]v_x v_{xx} + h_x + h_v v_x,\qquad (x,t)\in D(t_0-\delta, t_0+\delta).
$$
For any given $(x_1, t_1)$ with $x_1 :=r(t_1)$, noting $B(0)=0$ and the boundedness of $v_{xxx}$,
by taking limit as $(x,t)\to (x_1,t_1)$ in this equation we have
$$
\lim\limits_{(x,t)\to (x_1, t_1)} v_{tx} = H_3 (t_1) := [B'(0)+2]H_1 (t_1) H_2 (t_1) + h_x(x_1,0) + h_v(x_1,0) H_1 (t_1).
$$
Hence, $v_{tx}$ is continuous in $\overline{D^+}$. Differentiating Darcy's law
$r'(t) = -v_x(r(t)-0,t)$ and using the continuity of $v_x, v_{xx}, v_{xt}$ in $\overline{D^+}$ we conclude that
$$
r''(t) = v_x(r(t)-0,t) v_{xx} (r(t)-0,t) - v_{xt}(r(t)-0,t) = H_1(t) H_2(t) - H_3(t),
$$
and it is continuous in $t\in [t_0-\delta, t_0+\delta]$.
\end{proof}

\end{document}